\documentclass{article}
\usepackage{amsmath,amssymb,amsthm}

\title{Automorphisms of the truth-table degrees are fixed on a cone}
\author{Bernard A. Anderson \thanks{The author was partially supported by two projects of the Ministry of Education of the Czech Republic: 
Research project MSM0021620838 and Institute for Theoretical Computer Science project 1M0545.}\\Department of Theoretical Computer\\Science 
and Mathematical Logic\\Faculty of Mathematics and Physics\\Charles University\\Malostransk\'e N\'am\v{e}st\'i 25\\118 00 Prague 1, 
Czech Republic}

\newtheorem{theorem}{Theorem}[section]

\newtheorem{cor}[theorem]{Corollary}

\newtheorem*{claim}{Claim}
\theoremstyle{definition}
\newtheorem{definition}{Definition}

\newcommand{\setsep}{\ensuremath{\, | \;}}
\newcommand{\reals}{2^\omega}
\newcommand{\strings}{2^{<\omega}}
\newcommand{\num}{\in\omega}

\newcommand{\conv}{\!\downarrow}
\newcommand{\dive}{\!\uparrow}

\newcommand{\Dtt}{\ensuremath{D_{tt}}}
\newcommand{\ltt}{\ensuremath{\leq_{tt}}}
\newcommand{\gtt}{\ensuremath{\geq_{tt}}}
\newcommand{\ett}{\ensuremath{\equiv_{tt}}}
\newcommand{\half}{\ensuremath{\frac{1}{2}}}
\newcommand{\concat}{\: \hat{\ } \:}
\newcommand{\rstrd}{\mathrel{\mbox{\raisebox{.5mm}{$\upharpoonright $}}}} 
\newcommand{\smrstrd}{\mathrel{\mbox{\raisebox{.2mm}{$\scriptstyle \upharpoonright $}}}} 

\begin{document}
\maketitle

\begin{abstract}
Let \Dtt\ denote the set of truth-table degrees.  A bijection $\pi : \Dtt \to \Dtt$ is an automorphism if for all 
truth-table degrees $x$ and $y$ we have $x \ltt y \Leftrightarrow \pi (x) \ltt \pi (y)$.  We say an automorphism $\pi$ 
is fixed on a cone if there is a degree $b$ such that for all $x \gtt b$ we have $\pi (x) = x$.  We first prove that 
for every 2-generic real $X$ we have $X^\prime \not\leq_{tt} X \oplus 0^\prime$.  We next prove that for every real 
$X \gtt 0^\prime$ there is a real $Y$ such that $Y \oplus 0^\prime \ett Y^\prime \ett X$.  Finally, we use this to 
demonstrate that every automorphism of the truth-table degrees is fixed on a cone.
\end{abstract}

\section{Introduction}

The structure of the Turing degrees and the degrees of other reducibilities is a basic object in Computability 
Theory.  An important tool in studying the differences between degree structures is examining the properties of 
their automorphisms.  We begin with several definitions.  We view reals as elements of the Cantor space (infinite 
binary strings).

\begin{definition} A real $X$ is $n$-generic if for every $\Sigma_n$ set $S$ of strings either there is a 
number $l$ such that $X \rstrd l \in S$ or there is a number $l$ such that for every string $\tau \supseteq X \rstrd l$ we have 
$\tau \notin S$. \end{definition}

\begin{definition} $\Phi$ is a truth-table reduction if there is a computable function $f: \omega \to \omega$ such that 
for all $n \num$ and strings $\sigma$ of length $f(n)$ we have $\Phi^\sigma (n) \conv$.  We define $A \ltt B$ if there is 
a truth-table reduction $\Phi$ such that $\Phi (B) = A$. \end{definition}

\begin{definition} $A \leq_m B$ if there is a computable function $f:\omega \to \omega$ such that for all $n \num$ we have 
$x \in A \Leftrightarrow f(x) \in B$\@.  This is called many-one reducibility.  \end{definition}

\begin{definition} $A \leq_h B$ if $A$ is hyperarithmetic in $B$.  \end{definition}

We let $D$ be the set of Turing degrees and define \Dtt, $D_m$, and $D_h$ similarly.  We use $\leq_r$ and $D_r$ to 
denote an arbitrary reducibility and the corresponding set of degrees.  An automorphism of a degree structure is a bijection 
(a one-to-one and onto function) which preserves the order.

\begin{definition} An automorphism of $D_r$ is a bijection $\pi : D_r \to D_r$ such that for all degrees $x$ and $y$ 
we have $x \leq_r y \Leftrightarrow \pi (x) \leq_r \pi (y)$. \end{definition}

We say an automorphism $\pi$ is fixed on an (upper) cone if there is a degree $b$ such that for all $x \geq_r b$ we have
$\pi (x) = x$.  We call $b$ the base of the cone.

The truth-table degrees are often used in Computability Theory, especially in the areas of randomness and measure.  However, 
almost nothing is known about the automorphisms of the truth-table degrees.  Kjos-Hanssen \cite{Bjorn} has recently used 
lattice embeddings to obtain a result in a related area.  He showed that in the structure of the truth-table degrees with jump, 
($\Dtt$, $\ltt$, $^\prime$), every automorphism is fixed on the cone with base $0^{(4)}$.

In this paper we prove that every automorphism $\pi$ of the truth-table degrees is fixed on a cone.  The base of the cone is 
given by $d^{\prime \prime} \oplus \pi (d^{\prime \prime}) \oplus e^{\prime \prime} \oplus \pi^{- 1} (e^{\prime \prime})$ 
where $d = \pi^{-1} (0^{\prime \prime \prime})$ and $e = \pi (0^{\prime \prime \prime})$.  To show this, we prove that for 
every real $X \gtt 0^\prime$ there exists a real $Y$ such that $Y \oplus 0^\prime \ett Y^\prime \ett X$.  We will also prove 
that for every 2-generic real $X$ we have $X^\prime \not\leq_{tt} X \oplus 0^\prime$. 

To better understand the automorphisms of the truth-table degrees, we examine the automorphisms of other degree structures.  
From results known so far, it seems that the stronger the reducibility, the fewer the restrictions on the automorphisms of its 
degree structures.  At one extreme, the hyperdegrees are rigid; there is no nontrivial automorphism of $D_h$ \cite{hdegrees}.  

At the other end, there are very few restrictions on the automorphisms of the many-one degrees.  Odifreddi has proved that $0$ 
is the only $m$-degree fixed by every automorphism \cite{Odifreddi}.  Shore has shown that there are $2^{2^\omega}$ many 
automorphisms of the $m$-degrees \cite{Odifreddi}.  He has also noted that there is an automorphism of the $m$-degrees which is not 
fixed on any cone \cite{mdegrees}.  Combined with our main result, this provides a tangible difference between the 
automorphisms of the $m$-degrees and those of the $tt$-degrees.  

Our expectation is that the automorphisms of the $tt$-degrees will be more restricted than those of the $m$-degrees, but less 
restricted than those of the Turing degrees.   

There are several theorems which limit the possibilities for automorphisms of the Turing degrees.  Nerode and Shore 
\cite{cone} proved every automorphism of the Turing degrees is fixed on a cone.  Slaman and Woodin \cite{auto} 
improved this result by showing every automorphism is fixed on the cone with base $0^{\prime \prime}$.  They also proved 
that if $y$ is a Turing degree containing a 5-generic real and $\pi$ and $\rho$ are automorphisms such that 
$\pi (y) = \rho (y)$, then $\pi = \rho$.  Finally, Shore and Slaman \cite{jump} proved the Turing jump is definable in $D$, 
so $\pi (x^\prime) = (\pi(x))^\prime$ for every Turing degree $x$ and automorphism $\pi$.

The fundamental question ``Is there a nontrivial automorphism of the Turing degrees?'' is generally considered to be open as 
of this writing.  Cooper \cite{Cooper} has given an affirmative answer, but the proof has not yet been verified by leading 
experts.  Slaman and Woodin \cite{auto} have proved that the set of automorphisms is at most countable, and that the statement 
``There is a nontrivial automorphism of the Turing degrees'' is absolute between well-founded models of ZFC.

The author would like to thank Theodore Slaman and Antonin Ku\v{c}era for several helpful discussions.

\section{Background}

We wish to convert to the truth-table degrees the theorem of Nerode and Shore \cite{cone} that every automorphism of the Turing 
degrees is fixed on a cone.  Nerode and Shore proved truth-table and many-one analogues for many of their theorems on 
automorphisms.  They noted ``The general pattern of generalization is that if a proof for the Turing degrees does not use some form of the 
Friedberg Completeness Theorem then some version or other of the result holds for all of these reducibilities'' \cite{cone}.  In 
1984, Mohrherr \cite{ttdegrees} observed more specifically that proving every automorphism of the truth-table degrees is fixed on 
a cone required strong jump inversion of the truth-table degrees; for every real $X \gtt 0^\prime$ there is a real $Y$ such that 
$Y \oplus 0^\prime \ett Y^\prime \ett X$\@. 

Mohrherr \cite{ttdegrees} adapted Friedberg's proof of strong jump inversion of the Turing degrees to prove ordinary jump 
inversion of the truth-table degrees; for every real $X \gtt 0^\prime$ there is a real $Y$ such that $Y^\prime \ett X$\@.  
Friedberg's proof obtained strong jump inversion by using the fact that every 1-generic real $X$ is such that 
$X^\prime \leq_T X \oplus 0^\prime$.  To adapt this part of Friedberg's proof, we might hope to prove an analogous fact for 
the truth-table degrees; for every sufficiently generic real $X$, the statement $X^\prime \ltt X \oplus 0^\prime$ holds.  
Somewhat unexpectedly, we instead find that for every sufficiently generic real $X$, the statement fails.

\section{Strong jump inversion}

We prove that for every 2-generic real $X$ we have $X^\prime \not\leq_{tt} X \oplus 0^\prime$.  An intuitive explanation is 
that this holds because a truth-table computation of $X^\prime$ is bounded in its use of $X$, but the following is dense: the $n$ 
required such that $\{e\}^{X \smrstrd n}(e) \conv$ exceeds any fixed bound computable in $e$.

We give an outline of a proof by contradiction.  We suppose there is a truth-table reduction $\Phi$ with bound $f$ such that 
$\Phi (X \oplus 0^\prime) = X^\prime$ for some 2-generic real $X$\@.  We assume $f(n)$ is even for all $n$.  By genericity, there is 
an $l \num$ such that for every $\tau \supseteq X \rstrd l$ we never have 
$\Phi^{(\tau \oplus 0^\prime)\smrstrd f(n)} (n) = 0$ and $\{n\}^\tau (n) \conv$ for any $n$.  We can then use the Recursion Theorem to 
find an $m$ such that for any $\sigma$ we have $\{m\}^\sigma (m) \conv \Leftrightarrow \sigma (\half f(m) + 1) = 0$.  But the computation of 
$\Phi^{(\tau \oplus 0^\prime)\smrstrd f(m)} (m)$ can only use the first $\half f(m)$ many digits of $\tau$.  Hence, if $m$ is 
sufficiently large and $X(\half f(m) +1) = 1$ then $(X \rstrd \half f(m))\concat 0$ contradicts the statement given by genericity.

\begin{theorem} Let $X$ be a 2-generic real.  Then $X^\prime \not\leq_{tt} X \oplus 0^\prime$. \label{genlow} \end{theorem}
\begin{proof} Suppose not.  Let $\Phi$ be a truth-table reduction with bound $f$ such that 
$\Phi (X \oplus 0^\prime) = X^\prime$.  Without loss of generality, let $f(n)$ be even for all $n$.  

We define a set $S$ of strings $\sigma$ which witness a failure of $\Phi (A \oplus 0^\prime) = A^\prime$ for every 
$A \supseteq \sigma$.
$$S = \{ \sigma \in \strings \setsep \exists n [\Phi^{(\sigma \oplus 0^\prime)\smrstrd f(n)} (n) = 0 \ \wedge\ \{n\}^\sigma (n) \conv] \}.$$
$X$ does not meet $S$ because $\Phi (X \oplus 0^\prime) = X^\prime$.  We note $S$ is $\Sigma_1 (0^\prime)$ so $S$ is $\Sigma_2$.  
Since $X$ is 2-generic, there is a $l$ such that for every $\tau \supseteq X \rstrd l$ we have $\tau \notin S$.

Define $\{j(y)\}^\sigma (n) = 0$ if $\sigma (\half f(y) +1) = 0$ and $\{j(y)\}^\sigma (n) \dive$ otherwise.  By the Recursion Theorem, 
there is an infinite computable set $M$ such that for every $m \in M$ we have $\{j(m)\} = \{m\}$.  

\begin{claim} $\exists v \in M$ such that $\half f(v) > l$ and $X(\half f(v) +1) = 1$. \end{claim}
\begin{proof} Let $V = \{\sigma \in \strings \setsep \exists m \in M \, [\half f(m) >l \ \wedge\ \sigma (\half f(m) + 1) = 1]\}$.  
Suppose $X$ does not meet $V$.  Then since $X$ is 2-generic there is a $k$ such that for every $\gamma \supseteq X \rstrd k$ we 
have $\gamma \notin V$.  But $(X \rstrd k)\concat 1^j \in V$ for any sufficiently large $j$, for a contradiction.  Therefore we 
have $X \rstrd k \in V$ for some $k$.  We then let $v$ be the witness that $X \rstrd k \in V$.
\end{proof}

Let $v$ be given by the claim.  Since $v \in M$ and $X(\half f(v) +1) = 1$ we have $\{v\}^X (v) = \{j(v)\}^X (v) = \,\uparrow$.  Hence 
$v \notin X^\prime$.  Thus $\Phi^{(X \oplus 0^\prime)\smrstrd f(v)} (v) = 0$.  Let $\tau = (X \rstrd \half f(v))\concat 0$.  Since 
$\half f(v) > l$ we have $\tau \supseteq X \rstrd l$ so $\tau \notin S$.  We already have $\Phi^{(X \oplus 0^\prime)\smrstrd f(v)} (v) = 0$ 
so we must conclude $\{v\}^\tau (v) \dive$.  But since $v \in M$ and $\tau (\half f(v) +1)=0$ we have 
$\{v\}^\tau (v) = \{j(v)\}^\tau (v) = 0$ for a contradiction.
\end{proof}

Kjos-Hanssen \cite{Bjorn2} has noted that Mohrherr's \cite{ttdegrees} construction for $0^\prime$ yields a 1-generic real $X$ such that 
$X^\prime \ltt X \oplus 0^\prime$.  Hence this result is sharp.

The proof of Theorem \ref{genlow} illustrates some of the difficulties in proving strong jump inversion for the truth-table degrees.  
When constructing a real $Y$ such that $Y^\prime \ltt Y \oplus 0^\prime$ we must keep making choices about the values that 
$Y^\prime$ will take, before we are able to extend $Y$ to confirm these choices.  As a result, we need a way to find choices which 
maximize the possibilities available to us later.  Our original proof did this by using a ``bushy'' structure similar to those developed 
by Kumabe and Lewis \cite{KumLew}.  Ku\v{c}era \cite{PAidea} suggested that switching to a PA (binary valued diagonally noncomputable) 
approach would greatly simplify the proof's bookkeeping.

Following the notation in \cite{DNR}, we define PA $= \{f \in \reals \setsep \forall x [f(x) \not= \{x\}(x)]\}$.  Ku\v{c}era and 
Slaman \cite{DNR} have shown how PA can be used as a ``universal'' $\Pi^0_1$ class which can easily be used to code information.  We start 
with some definitions.  Let $f \in \reals$, $M$ be an infinite set, and $\langle m_i \setsep i \num \rangle$ be an increasing enumeration 
of $M$.  We define Restr $(f,M)$ to be the function $g$ where $g(i) = f(m_i)$.  If we think of $M$ as a set of locations where information 
is coded, then Restr $(f,M)$ is the real coded into $f$.  For $\sigma \in \strings$ we define Restr $(\sigma, M)$ as a string $\tau$ where 
$\tau (i) = \sigma (m_i)$ for all $i$ where $m_i$ is less than the length of $\sigma$.  Finally, for $B \subseteq \reals$ we define 
Restr $(B,M) = \{g \in \reals \setsep \exists f \in B \, [g = \mbox{Restr}(f,M)]\}$.

The following theorem allows us to uniformly shrink a $\Pi^0_1$ subclass while still recursively coding information. 

\begin{theorem}[Ku\v{c}era and Slaman \cite{DNR}] Let $B \subseteq$ PA be a $\Pi^0_1$ class.  Then there is an infinite computable set $M$ 
such that if $B$ is nonempty then Restr $(B,M) = \reals$.  Moreover, we can (uniformly) computably find an index for $M$ from an index 
for $B$. \label{KSresult} \end{theorem}

For $T \subseteq \strings$ a tree, let $[T]$ denote the set of paths through $T$.  For an index of a 
$\Pi^0_1$ tree $T$ we use a $n$ such that $T = \strings \setminus W_n$ (we view $W_n$ as the $n$-th c.e.\ set of strings).  An index of a 
$\Pi^0_1$ class $[T]$ is an index for $T$.  For a tree $T$ and a string $\sigma$ we define 
$T[\sigma] = \{\tau \in T \setsep \tau \supseteq \sigma\ \vee\ \tau \subseteq \sigma \}$ to be the part of $T$ compatible with the root 
$\sigma$.

We outline the proof of strong jump inversion for the truth-table degrees.  Given $X \gtt 0^\prime$ we construct a $Y$ such that 
$X \ltt Y \oplus 0^\prime$ and $Y^\prime \ltt X$.  We build $Y$ as a path through a shrinking set of trees $T_i$.  If $i \in A^\prime$ for 
every path $A \in [T_i]$ then we let $T_{i+1} = T_i[Y_i]$.  If not, we obtain $T_{i+1}$ by removing from $T_i[Y_i]$ every point which forces 
$i \in Y^\prime$.  In either case, we force a value for $Y^\prime (i)$.  

The key idea is that Theorem \ref{KSresult} guarantees that the resulting tree will be thick enough to code with.  In this way the Theorem 
internalizes the role of a ``bushy'' condition which could be used to ensure that the chosen tree is sufficiently thick (as in \cite{KumLew}).  
We can then extend $Y$ to record the value of $X(i)$ at the next coding location and proceed to the next stage.  To prove that we can calculate 
$X$ or $Y^\prime$ by truth-table reductions, we need to obtain a computable bound on the use of $Y \oplus 0^\prime$ or $X$ needed to reach a 
given stage of the construction.  We do this by exhaustively considering all possible values for the inputs and parameters. 

\begin{theorem} Let $X$ be a real such that $X \gtt 0^\prime$.  Then there exists a real $Y$ such that 
$Y \oplus 0^\prime \ett Y^\prime \ett X$. \label{main} \end{theorem}
\begin{proof} We construct the real $Y$ in stages.  For bookkeeping, we will also use $\Pi^0_1$ trees $T_i$ for $i \num$.  At each stage we 
will have $Y_{i+1}$ extend $Y_i$, $T_{i+1} \subseteq T_i$, and $Y_i \in T_i$.  We start with $Y_0 = \langle \rangle$ and $T_0 =$ PA.

At stage $i$ we use $0^\prime$ to determine if there is a number $d$ such that for all strings $\sigma$ of length $d$ extending $Y_i$ we have 
$\sigma \notin T_i$ or $\{i\}^\sigma (i) \conv$.  This is a $\Sigma^0_1$ question.  If the answer is yes we let $T_{i+1} = T_i[Y_i]$.  We note 
that since we will have $Y \in [T_{i+1}]$, we have forced $i \in Y^\prime$.  

If the answer is no we let $T_{i+1} = T_i[Y_i] \setminus \{\sigma \in \strings \setsep \{i\}^\sigma (i) \conv \}$.  We note that 
$T_{i+1}$ is $\Pi^0_1$ and we have forced $i \notin Y^\prime$. 

In either case, by compactness and our choice of $Y_i$, we have that $[T_{i+1}]$ is nonempty.  Let $M$ be the set given by 
Theorem \ref{KSresult} for $[T_{i+1}]$ and let $m$ be the least element of $M$.  We then use $0^\prime$ to define $Y_{i+1}$ to be the leftmost 
string $\sigma \in T_{i+1}$ of length $m$ such that $\sigma (m) = X (i)$ and $\big[T_{i+1}[\sigma]\big]\not= \emptyset$.  Such a $\sigma$ 
exists by our choice of $m$.  This completes stage $i$.

It is clear that this construction can be done computably in $X$ (since $X \gtt 0^\prime$).  For every $i \num$ we can determine if 
$i \in Y^\prime$ from the $i$-th stage of the construction.  Hence $Y^\prime$ is computable in $X$.  We can also see that given $Y$ 
and $0^\prime$ we can use the construction to find $X$.  At the end of the $i$-th stage we read $Y(m)$ to find $X(i)$.  Thus 
$Y \oplus 0^\prime \equiv_T Y^\prime \equiv_T X$.  

To obtain truth-table reductions we must use computations which are total.  If an unexpected input is used for $0^\prime$ in the 
computation of $Y$ then at some stage the computation may find no possible extension for $Y_{i+1}$.  To fix this, we define 
the computation so that it ceases normal operations and outputs an infinite string of 0's if this occurs.  Since at any stage there 
are only finitely many possibilities for $Y_{i+1}$, the computation will realize an error state has been reached.  We define
the computation of $X$ on unexpected inputs for $Y$ or $0^\prime$ similarly.  It remains to show we can bound the use on these 
computations (for expected and unexpected inputs).

We first define several computable functions.  Let $m$ be the index of a $\Pi^0_1$ tree $T$\@.  For $i \num$ we define 
$u(m,i)$ to be the index of $T \cap \{\tau \in \strings \setsep \{i\}^\tau (i) \dive \}$.  For $\sigma$ a string we let $q(m,\sigma)$ 
be the index of $T[\sigma]$.  Using Theorem \ref{KSresult}, let $s$ be a computable function such that if $i$ is the index for a tree 
then $s(i)$ is the least element of the corresponding set $M$.  We note that Theorem \ref{KSresult} still gives the index of a 
(meaningless) total, infinite computable set if the $\Pi^0_1$ class is empty.  Hence $s$ is total.

We define computable functions $t$ and $l$ to bound the indices of the $T_i$ and the lengths of the $Y_i$, respectively.  These 
functions are defined by simultaneous induction starting with $t(0)$ equal to the index of PA and $l(0) = 0$.  For the inductive 
step, we define $t(i+1)$ and $l(i+1)$ to be the largest possible result for any values used in the roles of $Y^\prime \rstrd i$, 
$X \rstrd i$, and for $j \leq i$, $T_j$ and $Y_j$.  Hence we first define
$$t(i+1) = \max_{\sigma \in 2^{\leq l(i)},\ e \leq t(i)}\Big(\max(q(e,\sigma),q(u(e,i),\sigma))\Big).$$  We then let 
$l(i+1) = \max_{e \leq t(i+1)} (s(e))$ to complete the induction.

We next wish to find a computable function $g$ to bound the amount of $0^\prime$ used in steps of the construction.  Let 
$h: \strings \times \omega \times \omega \to \omega$ be a computable function such that $h(\tau,n,i) \in 0^\prime$ iff there is a 
number $d$ such that for all strings $\sigma$ of length $d$ extending $\tau$ we have $\sigma \in W_n$ or $\{i\}^\sigma (i) \conv$.  
Similarly, let $j: \strings \times \omega \to \omega$ be a computable function such that $j(\tau,n) \in 0^\prime$ iff there is a 
number $d$ such that for all strings $\sigma$ of length $d$ extending $\tau$ we have $\sigma \in W_n$.  We then let 
$$g(i) = \max_{\tau \in 2^{\leq l(i+1)},\ m \leq i,\ n \leq t(i+1)} \Big(\max(h(\tau,n,m),j(\tau,n))\Big)$$
so the construction uses at most $0^\prime \rstrd g(i)$ for step $i$.  

We can now complete the proof.  In calculating $Y^\prime (n)$ from $X$ (and $0^\prime$), we need to reach the $n$-th step of the 
construction.  This requires the first $n$ digits of $X$ and $0^\prime \rstrd g(i)$.  Hence if $f$ witnesses that $0^\prime \ltt X$, then we need 
$X \rstrd \max(n,f(g(n)))$ to calculate $Y^\prime (n)$.  Thus we have a computable bound for the use of $X$.  Similarly, 
in calculating $X(n)$ from $Y$ and $0^\prime$ we need $Y \rstrd l(n+1)$ and $0^\prime \rstrd g(n)$, so we have computable bounds on the use of 
$Y$ and $0^\prime$.

These bounds also work for unexpected inputs since they use the largest result for all possible values of the inputs for $X$, $Y$, and 
$0^\prime$.  If an error state is reached then the use no longer increases, so it remains below the bounds.  Therefore, 
$Y \oplus 0^\prime \ett Y^\prime \ett X$. 
\end{proof}

We wish to prove a relativized form of this theorem; given reals $A$ and $X \gtt A^\prime$ there is a real $Y \gtt A$ such that 
$X \ett Y^\prime \ett Y \oplus A^\prime$.  The difficulty is showing that the bounding functions for the truth-table 
reductions are computable (instead of just computable in $A$).  To do this we use a powerful relativized form of Theorem 
\ref{KSresult}

\begin{theorem}[Ku\v{c}era and Slaman \cite{DNR}] Let $A$ be a real and $B \subseteq$ PA($A$) be a $\Pi_1^{0, A}$ class.  Then there 
is an infinite computable set $M$ such that if $B$ is nonempty then Restr $(B,M) = \reals$.  Moreover, an index of $M$ can be 
found (uniformly) computably from an index of B, i.e.\ in a uniform way which does not depend on the oracle $A$\@. \label{KSrel} 
\end{theorem}

We also use the relativized $s$-$m$-$n$ Theorem.

\begin{theorem}[\cite{Soare}] Let $m,n \geq 1$.  There exists an injective computable function $s^m_n$ such that for all reals $A$ and 
all $x, y_1, \ldots , y_m \num$ we have 
$$\{ s^m_n (x,y_1, \ldots, y_m)\}^A (z_1, \ldots , z_n) = \{x \}^A (y_1, \ldots , y_m, z_1, \ldots, z_n).$$ \label{smn} \end{theorem}

We can now prove the relativized form of Theorem \ref{main}.

\begin{theorem} Let $A$ and $X$ be reals such that $X \gtt A^\prime$.  Then there is a real $Y \gtt A$ such that 
$X \ett Y^\prime \ett Y \oplus A^\prime$. \label{relmain} \end{theorem}
\begin{proof} Following the approach in \cite{DNR}, we let $C = \{f \in \mbox{PA}(A) \setsep \mbox{Restr}(f,M) = A \}$ where $M$ is from 
Theorem \ref{KSresult} applied to PA($A$).  $C$ is a nonempty $\Pi^{0,A}_1$ class such that for every $f \in C$ we have $f \gtt A$ (we 
check the values of $f$ on the recursive set $M$)\@.  

We begin the construction with $T_0 = C$\@.  This ensures that $Y \gtt A$ (so also $Y^\prime \gtt A^\prime$)\@.  The rest 
of the construction remains essentially the same, changing the oracle $0^\prime$ to $A^\prime$ and using Theorem \ref{KSrel} 
instead of Theorem \ref{KSresult}.  We do not change $\{ i \}^\sigma$ to $\{ i \}^A$ since we wish to control $Y^\prime$, not 
$A^\prime$\@.  For the computation of the bounding functions, we note that $s$ is computable by Theorem \ref{KSrel}.  The functions 
$u$, $q$, $h$, and $j$ are computable by Theorem \ref{smn}.
\end{proof}      

\begin{cor} Let $n \num$ and let $X$ and $A$ be reals such that $X \gtt A^{(n)}$.  Then there exists a real $Y \gtt A$ such that 
$Y \oplus A^{(n)} \ett Y^{(n)} \ett X$. \label{mcor} \end{cor}
\begin{proof} We prove the statement by induction.  The base case is given by Theorem \ref{relmain}.  For the inductive 
case, assume the statement holds for $n$ and let $X \gtt A^{(n+1)}$ be arbitrary.  By Theorem \ref{relmain} with base $A^{(n)}$ 
there exists a $Z \gtt A^{(n)}$ such that $Z \oplus A^{(n+1)} \ett Z^\prime \ett X$.  Applying the inductive 
hypothesis to $Z$ we obtain a real $Y \gtt A$ such that $Y \oplus A^{(n)} \ett Y^{(n)} \ett Z$.  We then note 
$Y \oplus A^{(n+1)} \gtt Z \oplus A^{(n+1)} \gtt X \gtt Z^\prime \gtt Y^{(n+1)}$.  Hence 
$Y \oplus A^{(n+1)} \ett Y^{(n+1)} \ett X$, completing the induction.  \end{proof}

\section{Applications to automorphisms}

To apply strong jump inversion to automorphisms of the truth-table degrees, we need a way to code a real in the truth-table 
degrees using a finite number of parameters.  We can use the method of Nerode and Shore \cite{cone} to do this, but obtain 
a slightly lower base for the cone by using a theorem of Mytilinaios and Slaman \cite{ttcoding}.  

The theorem below gives a finite set of parameters $\vec{p}$ and a countable set of reals $G_1, G_2, \ldots$ representing the 
natural numbers.  These reals satisfy equations which allow us to uniquely determine $G_{i+1}$ given $\vec{p}$ and $G_i$.  
Hence $\vec{p}$ and $G_1$ uniquely determine every $G_i$.  Since the relations depend only on $\ltt$, they are invariant under 
automorphisms of the truth-table degrees.

Let $(X)$ denote the ideal generated by $X$ (the $tt$-degrees below $X$).  

\begin{theorem}[Mytilinaios and Slaman \cite{ttcoding}] Let $B \gtt 0^\prime$.  Then there exist reals 
$\vec{p} = \langle E_1, E_2, D_1, D_2, F_1, F_2 \rangle \ltt B^{\prime \prime}$ and 
$\langle G_n \setsep n \num \rangle \ltt B^{\prime \prime}$ (uniformly) such that:
\begin{enumerate}
\item{B is $tt$-computable in each component of $\vec{p}$ and $B \ltt G_n$ for all $n \num$.}
\item{For any $G_{n_1}, G_{n_2}, \ldots G_{n_k}$ and $m \not= n_j$ for all $j<k$ we have\\$(G_{n_1} \oplus \ldots \oplus G_{n_k}) \cap (G_m) = (B)$.}
\item{$D_1 \not\geq_{tt} D_2$ and for any $n \num$ we have $D_1 \oplus G_n \gtt D_2$.}
\item{For $n$ odd, $(F_1 \oplus G_n) \cap (E_1) = (G_{n+1})$ and \\For $n$ even, $(F_2 \oplus G_n) \cap (E_2) = (G_{n+1})$.} 
\end{enumerate} \label{slamyt1} \end{theorem}

In Mytilinaios and Slaman \cite{ttcoding} this theorem is given with the value of $0^\prime$ for $B$ (the variable $O$ is used in 
the paper).  However, it is clear that the proof works for any $B \gtt 0^\prime$.  The complexity bounds used in the construction 
are given in terms of $B$ and the construction makes no use of any property of $B$ except an implied use of the fact that 
$B \gtt 0^\prime$ (so the construction can enumerate all $tt$-reductions).

We use Theorem \ref{slamyt1} with the next theorem to code a real $S$.  Given $S$ we find a real $Q$ such that $G_i$ satisfies a 
certain relation involving $Q$ if and only if $i \in S$.  Again, since the relation only depends on $\ltt$ it is invariant under 
automorphisms.

\begin{theorem}[Mytilinaios and Slaman \cite{ttcoding}] Let $B$, $\vec{p}$, and $\langle G_n \setsep n \num \rangle$ satisfy the 
conditions above.  Let $S \subseteq \omega$.  Then $S$ is $\Sigma^0_2 (\vec{p}) \Leftrightarrow \exists Q [Q \ltt \vec{p} \ \wedge\\
\forall i \num \,[i \in S \leftrightarrow \exists X [ X \ltt G_i \ \wedge\ X \ltt Q \ \wedge\ D_2 \ltt X \oplus D_1]]]$.\label{slamyt2} \end{theorem}

We can now outline the proof that every automorphism of the truth-table degrees is fixed on a cone.  We say $q$ codes $S$ for $\vec{p}$ and 
$\langle G_n \setsep n \num \rangle$ over base $B$ if $q$ satisfies the conditions for the set witnessing the result of Theorem \ref{slamyt2} for 
the set $S$, when using $\vec{p}$, $\langle G_n \setsep n \num \rangle$, and $B$ for the sets satisfying the conditions from Theorem \ref{slamyt1}.  
We omit the base when it is clear from context.  We note that if $Q$ codes a real $R$ for $\vec{p}$ and $\langle G_n \setsep n \num \rangle$ over base 
$B$, and $R$ is such that $R \ett B^{\prime \prime}$, then $\vec{p} \ltt R$.  Furthermore $R$ is $\Sigma^0_2 (\vec{p})$, so 
$R \ltt (\vec{p})^{\prime \prime}$.

Suppose we let $R$ in the above example have degree $\pi (y)$ for an automorphism $\pi$ and some sufficiently high $y$.  Then $\vec{p} \ltt \pi (y)$ 
so $\pi^{-1} (\vec{p}) \ltt y$.  Since the relations are invariant under automorphisms, $\pi^{-1} (q)$ codes $R$ for $\pi^{-1} (\vec{p})$ and 
$\langle \pi^{-1} (G_n) \setsep n \num \rangle$.  Hence $R \ltt (\pi^{-1} (\vec{p}))^{\prime \prime}$, as above.  Thus $R \ltt y^{\prime \prime}$ so 
$\pi (y) \ltt y^{\prime \prime}$.  

We next apply strong jump inversion.  Let $d$ be some sufficiently high base degree and let 
$x \gtt d^{\prime \prime} \oplus \pi (d^{\prime \prime})$.  By strong jump inversion relative to $d$, there is a $y \gtt d$ such that 
$x \ett y^{\prime \prime} \ett y \oplus d^{\prime \prime}$.  By the above paragraph, $\pi (y) \ltt y^{\prime \prime}$.  We then have 
$\pi (x) \ett \pi (y \oplus d^{\prime \prime}) \ett \pi(y) \oplus \pi (d^{\prime \prime}) \ltt y^{\prime \prime} \oplus x \ltt x$.  
We complete the proof by symmetry.  

We formalize this argument below.

\begin{theorem} Let $\pi : \Dtt \to \Dtt$ be an automorphism.  Then there is a degree $b$ such that for all degrees $x \gtt b$ we have 
$\pi (x) = x$. \end{theorem}
\begin{proof} Let $d = \pi^{-1} (0^{\prime \prime \prime})$ and $e = \pi (0^{\prime \prime \prime})$.  Let 
$b = d^{\prime \prime} \oplus \pi (d^{\prime \prime}) \oplus e^{\prime \prime} \oplus \pi^{-1} (e^{\prime \prime})$ and let 
$x \gtt b$ be arbitrary.

Since $x \gtt d^{\prime \prime}$, by Corollary \ref{mcor} with base $d$ there is a $y \gtt d$ such that 
$x \ett y^{\prime \prime} \ett y \oplus d^{\prime \prime}$.  We note $\pi (y) \gtt \pi (d) = 0^{\prime \prime \prime}$.  By 
Corollary \ref{mcor} with base $0^\prime$, there is a $Z \gtt 0^\prime$ such that $Z^{\prime \prime} \ett \pi (y)$.

Let $\vec{p} \ltt Z^{\prime \prime}$ and $\langle G_n \setsep n \num \rangle$ be given by Theorem \ref{slamyt1} with base $Z$.  
We note $Z^{\prime \prime}$ is $\Sigma^0_2 (Z)$ and $Z \ltt \vec{p}\,$, so $Z^{\prime \prime}$ is $\Sigma^0_2 (\vec{p})$.  By 
Theorem \ref{slamyt2}, let $q \ltt \vec{p}$ code $Z^{\prime \prime}$ for $\vec{p}$ and $\langle G_n \setsep n \num \rangle$.

Since $\ltt$ is preserved under automorphisms, $\pi^{-1} (q)$ codes $Z^{\prime \prime}$ for $\pi^{-1} (\vec{p})$ and 
$\langle \pi^{-1} (G_n) \setsep n \num \rangle$.  By Theorem \ref{slamyt2}, $Z^{\prime \prime}$ is $\Sigma^0_2 (\pi^{-1} (\vec{p}))$.

Since $\vec{p} \ltt Z^{\prime \prime} \ett \pi (y)$, we have $\pi^{-1} (\vec{p}) \ltt y$.  Hence $Z^{\prime \prime}$ is $\Sigma^0_2 (y)$.  
Thus $Z^{\prime \prime} \ltt y^{\prime \prime}$ so $\pi (y) \ltt y^{\prime \prime} \ltt x$ (since we chose $y$ and $Z$ such that 
$Z^{\prime \prime} \ett \pi (y)$ and $y^{\prime \prime} \ett x$).

We chose $x \gtt b$ so we have $\pi (d^{\prime \prime}) \ltt x$.  Hence $\pi (y) \oplus \pi (d^{\prime \prime}) \ltt x$ so 
$y \oplus d^{\prime \prime} \ltt \pi^{-1} (x)$.  We chose $y$ such that $y \oplus d^{\prime \prime} \ett x$ so we have 
$x \ltt \pi^{-1} (x)$.  Thus $\pi (x) \ltt x$.

We have used the fact that $x \gtt d^{\prime \prime} \oplus \pi (d^{\prime \prime})$ to prove that $\pi (x) \ltt x$.  By symmetry, we 
can use the fact that $x \gtt e^{\prime \prime} \oplus \pi^{-1} (e^{\prime \prime})$ to prove that $\pi^{-1} (x) \ltt x$ so 
$x \ltt \pi(x)$.  Therefore $\pi (x) = x$.
\end{proof}

We note that if $\pi$ is such that $\pi (0^{\prime \prime \prime}) = 0^{\prime \prime \prime}$ and $\pi (0^{(5)}) = 0^{(5)}$ then the 
base of the cone is $0^{(5)}$.  In particular, for the structure of the truth-table degrees with jump, $(\Dtt, \ltt, \prime)$, every 
automorphism is fixed on the cone with base $0^{(5)}$.  By using lattice embeddings to work with the structure $(\Dtt, \ltt, \prime)$ 
directly, Kjos-Hanssen \cite{Bjorn} proved every automorphism is fixed on the cone with base $0^{(4)}$. 

Many of the results of Slaman and Woodin \cite{auto} which follow from the fact that every automorphism of the Turing degrees is fixed 
on a cone can be modified for the truth-table degrees.  For example, let $\pi : \Dtt \to \Dtt$ be an automorphism and let $b$ be the 
base of the cone on which $\pi =$ id.  Then for any ideal $I$ of $\Dtt$ with $b \in I$, the restriction $\pi \rstrd I$ 
is an automorphism of $I$.  To prove this, we observe that for any $x \in I$ we have $x \oplus b \in I$ and $b \ltt x \oplus b$.  Hence 
$\pi (x) \ltt \pi (x \oplus b) = x \oplus b$ so $\pi (x) \in I$, and similarly for $\pi^{-1}$.

\section{Conclusion}

One route for further exploration of the automorphisms of the truth-table degrees, would be to prove truth-table degree analogues of some 
of the major results of Slaman and Woodin \cite{auto} on automorphisms of the Turing degrees.  The next major obstacle in doing this is 
finding a way to code antichains in the structure $\Dtt$ using a finite number of parameters.

Given a countable antichain of truth-table degrees $A = \langle A_n \in \Dtt \setsep n \num \rangle$, we wish to find a finite set of 
parameters $\vec{p} \in \Dtt$ and a formula $\psi$ in the language $(\Dtt, \ltt)$ such that $x \in A \Leftrightarrow \psi (\vec{p}, x)$.  
The parameters $\vec{p}$ should at least be arithmetic in $A$, and preferably $\Sigma^0_n (A)$ for some low $n$.  

The proof of this statement for the Turing degrees given in Slaman and Woodin \cite{auto} cannot be directly adapted.  The proof relies 
heavily on the fact that every Turing degree contains a real which is computable in any infinite subset of itself.  A corresponding 
fact is false for the truth-table degrees.  Kjos-Hanssen, Merkle, and Stephan \cite{thin}, building on a result of Miller \cite{Miller}, 
proved that a real $A$ computes a diagonally non-computable function if and only if there is a real $B \leq_T A$ such that there is no 
hyperimmune real $C \geq_{wtt} B$.  Hence for such a $B$, given any $D \ett B$ we can take an infinite $C \subseteq D$ thin enough so that 
$C$ is hyperimmune and conclude $D \not\leq_{tt} C$\@.

Finding a way to code antichains will likely result in a proof that the statement ``There is a nontrivial automorphism of the 
truth-table degrees'' is absolute between well-founded models of ZFC.  It will also provide a necessary step in using the methods of Slaman 
and Woodin \cite{auto} to prove truth-table degree analogues of their other results on automorphisms of the Turing degrees.

There are several other open questions we can consider.  Sacks \cite{SacksJI} proved that given a $\Sigma^0_2$ real $X \geq_T 0^\prime$, 
there is a c.e.\ real $Y$ such that $Y^\prime \equiv_T X$.  Does some form of Sacks jump inversion hold for the truth-table degrees: given 
a $\Sigma^0_2$ real $X \gtt 0^\prime$, is there a c.e.\ real $Y$ such that $Y^\prime \ett X$ ?  Can we find such a $Y$ if we replace the 
requirement that $Y$ be c.e.\ with $Y \ltt 0^\prime$ or with $Y$ is $\Delta^0_2$ ?  Finally, is $\leq_T$ definable in $(\Dtt, \ltt)$ ?

\nocite{*}
\bibliography{ttauto}
\end{document}